\documentclass[12pt]{article}
\usepackage{amsmath,latexsym, amsthm, amsfonts, amssymb, amsxtra,amscd, enumerate,color, cancel}
\usepackage{bbm, dsfont,latexsym} 
\usepackage{float,wrapfig}
\usepackage[english]{babel}
\usepackage[usenames,dvipsnames,svgnames,table]{xcolor}
\usepackage{times, graphicx,stmaryrd} 

\newtheorem{theorem}{Theorem}[section]
\newtheorem{remark}{Remark}[section]
\newtheorem{proposition}{Proposition}[section]
\newtheorem{corollary}{Corollary}[section]
\newcommand{\eqnsection}{
\renewcommand{\theequation}{\thesection.\arabic{equation}}
    \makeatletter
    \csname  @addtoreset\endcsname{equation}{section}
    \makeatother}
\eqnsection

\def\bint_#1^#2{\raise.5pt\hbox{\footnotesize${\displaystyle\int_{#1}^{#2}}$}}

\def\bsum_#1^#2{\raise.5pt\hbox{\footnotesize${\displaystyle\sum_{#1}^{#2}}$}}

\def\bfrac#1#2{\raise.5pt\hbox{\footnotesize$\dfrac{#1}{#2}$}}

\def\ch{\hbox{ch}}

\def\sh{{\rm sh}}

\def\sign{\hbox{sign}}

\def\bP{{\bold P}}
\def\bR{{\bold R}}

\def\bE{{\bold E}}

\def\cF{{\cal F}}
\def\cD{{\cal D}}

\def\cC{{\cal C}}
\def\cG{{\cal G}}

\def\ph{\varphi}

\def\ep{\varepsilon}

\def\ga{\gamma}
\def\si{\sigma}

\def\be{\beta}
\def\la{\lambda}
\def\om{\omega}

\def\del{\delta}
\def\wt{\widetilde}

\def\({(\hskip-1pt(}

\def \R{{\rm I}\mspace{-4mu}{\rm R}}
\def \B1{\mathbf{ 1}\mspace{-4.5mu}{\rm I}}

\def \BL{\mathbb{ L}}
\begin{document}

\title{On the local time process of a skew Brownian motion}

\author{Andrei Borodin\thanks{St. Petersburg State University, Russia, National Research University - Higher School of Economics, Russia, St. Petersburg Department of Steklov Mathematical Institute, Russia, E-mail: borodin@pdmi.ras.ru}\quad
and\quad Paavo Salminen\thanks{ Abo Akademi University, Faculty of Science and Engineering, Finland,
 E-mail: paavo.salminen@abo.fi}} 
\maketitle
\abstract{We derive  a Ray-Knight type theorem for the local time process (in the space variable) of a skew  Brownian motion up to an independent exponential time.  It is known that the local time seen as a density of the occupation measure and taken with respect to the Lebesgue measure has a discontinuity at the skew point (in our case at zero), but the local time taken with respect to the speed measure is continuous. In this paper we discuss this discrepancy by characterizing the dynamics of the local time process in  both these cases. The Ray-Knight type theorem is applied to study integral functionals of the local time process of the  skew Brownian motion. In particular, we  determine the distribution of the maximum of the local time process up to a fixed time,  which can be seen as the main new result of the paper.}\\



\thispagestyle{empty} \clearpage \setcounter{page}{1}

\section{Introduction}
\label{intro}

By a skew Brownian motion, SBM, with parameter $\beta\in(0,1)$ we mean a regular diffusion (in the sense of It\^o and McKean  \cite{itomckean74}, see also \cite{borodinsalminen15}) 
associated with the speed measure $m_\be$ and the scale function $S_\be$ given by
\begin{equation}
\label{e00}
m_\beta(dx) =\begin{cases}
4\beta dx, &x>0,\\
4(1-\beta)dx, &x<0,
\end{cases}
\end{equation}
and 
\begin{equation}
\label{e01}
S_\beta (x) =\begin{cases}
x/(2\beta), &x\geq0,\\
x/(2(1-\beta)), &x\leq0,
\end{cases}
\end{equation}
respectively. It is also assumed that $m_\beta(\{0\})=0.$ The notation $W_\beta=(W_\beta(t))_{t\geq 0}$ is used for SBM with parameter $\beta.$ Clearly, $W_{1/2}$ is a standard Brownian motion and, for simplicity, we omit in this case the index $1/2. $ From the general theory of diffusions we may deduce that the infinitesimal generator associated with the above scale and speed is  
\begin{align}
\label{e0}
\cG_\beta f(x)=\Big(\frac{d }{dm_\beta}\frac{d}{ dS_\beta}f\Big)(x)=\bfrac12\bfrac{d^2f}{dx^2}(x) , \quad  x\not= 0,&&\\
 \label{e00}\hskip-2cm
\cG_\beta f(0)=\bfrac{1}{2}f''(0+)=\bfrac{1}{2}f''(0-).\hskip2.3cm&&
 \end{align}
with the domain 
\begin{equation}
\label{e1}
 \cD=\{f:\ f,\ \cG_\beta f\in\cC_b(\bR),\ (1-\be) f'(0-)=\be f'(0+)\}.
 \end{equation}

SBM was introduced  by It\^o and Mckean in  \cite{itomckean74}  Problem 1, p. 115.  
To discuss this briefly, consider a reflected Brownian motion $|W|=(|W(t)|)_{t\geq 0}$ on $\R_+$ initiated from 0  and its excursions from 0 to 0.  A new sample path is constructed by multiplying a given  (positive) excursion of $|W|$ with $+1$ or $-1$ with some probability $p\in(0,1)$ and $1-p,$ respectively. It is assumed that the multiplications are done independently on different excursions. The resulting process is called a skew Brownian motion and  the problem posed  to the reader in  \cite{itomckean74} is to calculate its  scale function and speed measure. 

Recall also that $W_\beta$ is the unique strong solution of the stochastic equation (cf.  Harrison and Shepp \cite{harrisonshepp81})
$$
W_\be(t)=x+ W(t)+(2\be-1)\ell_\be(t,0),\qquad\qquad W_\be(0)=x, 
$$
where $W=(W(t))_{t\geq 0}$ is a given standard BM started from 0 and  $\ell_\be(t,0)$ is the local time of $W_\be$ at zero with respect to the Lebesgue measure, i.e., 
$$
 \ell_\be(t,0):=\lim_{\varepsilon\downarrow 0}\frac{1}{2\varepsilon} \int_0^t {\bf 1}_{(-\varepsilon, +\varepsilon)}(W_\beta(s))\,ds.
 $$

The main interest in this paper is in the local time process of $W_\beta.$ Our first task is to investigate the relationship between the local times normalized,  on the one hand, with respect to Lebesgue measure  and, on the other hand, with respect to speed measure $m_\beta.$ As shown in Walsh \cite{walsh78a}  in the first case the local time as a process in the space variable is discontinuous at 0. From the general diffusion theory it is known that in the second case the local time process is continuous at~ 0.  We clarify in the next section  this discrepancy by finding the multiplying constants connecting the different local times. In Ramirez et al. \cite{RTW15} the phenomenon is studied for general regular one-dimensional diffusions.   

In section 3 we present the Ray-Knight type theorems showing the diffusion structure of the local time process. From the excursion description of SBM one may anticipate that the diffusion structure is similar as in the case of a standard BM only difference coming from the distribution of the one sided local times at~0. This is indeed
 the case, and we give a proof in Appendix based on the Feynman-Kac formula.  However, we wish to underline that the approach introduced  in Ray \cite{ray63} and further developed  in Walsh \cite{walsh78}  Theorem 4.1  and McGill \cite{mcgill82} Theorem 3.3  cannot be ``directly'' applied to deduce the Ray-Knight type result for SBM since the scale function of SBM is not differentiable.  For a Ray-Knight type theorem for the local time at 0 of SBM viewed as a function of the initial state of SBM, see \cite{burdzychen01}.
 
 The final section contains the main result where we calculate the distribution of the maximum of the local time process. The corresponding distribution of standard BM is known, see Borodin \cite{borodin85} and Cs\'aki, F{\"o}ldes and Salminen \cite{csakifoldessalminen87}.  Our present work extends these results and formulas for SBM and can be applied to study, e.g., increments of the local time of SBM as is done for standard BM in the paper by Cs\'aki and  F{\"o}ldes   \cite{csakifoldes86}. Further potential applications may be found  in the research on the most favorite sites of a stochastic process, see  Bass and  Griffin \cite{bassgriffin85}, Shi and Toth \cite{shitoth00} and references therein. In fact,  our interest in the problem was arisen by Endre Cs\'aki and Antonia  F{\"o}ldes who in \cite{csakietal17} together with Cs{\"o}g{\"o} and R\'ev\'esz studied local times of BM on a multiray (also called a spider).

We refer to Lejay \cite{lejay06} for a survey on SBM and also for many references on the topic so far. SBM has found applications in modeling of different phenomena like  diffusion of a pollutant , brain imaging, and population ecology, for references see \cite{lejay06}. For more recent papers  we refer to Appuhamillage et al. \cite{appuhamillage11}  for  results on the joint distribution of occupation and local times and applications in the dispersion of a solute concentration, to  Lejay and  Pichot \cite{lejaypichot12} for exact simulation of SBM, to Lejay, Mordecki and Torres \cite{lejaymordeckitorres2014} for statistical aspects, and to Alvarez and Salminen \cite{alvarezsalminen17} and Rosello \cite{rosello12} for applications in financial mathematics.

\section{Continuity vs discontinuity}
\label{cont}

Our aim is to study the continuity of the local time viewed as a process in the space variable. For notational simplicity, we let $m:=m_\beta$ and $S:=S_\beta.$ The local time of $W_\beta$ with respect to the Lebesque measure is defined for all $x\in\R$ as 
$$
 \ell_\be(t,x):=\lim_{\varepsilon\downarrow 0}\frac{1}{2\varepsilon} \int_0^t {\bf 1}_{(x-\varepsilon, x+\varepsilon)}(W_\beta(s))\,ds,
 $$
where the limit exists a.s. We define similarly
$$
 L_\beta(t,x):= L(t,x):=\lim_{\varepsilon\downarrow 0}\frac{1}{m^*((x-\varepsilon, x+\varepsilon))} \int_0^t {\bf 1}_{(x-\varepsilon, x+\varepsilon)}(W_\beta(s))\,ds
 $$
where  
\begin{equation}
\label{mstar0}
m^*(dx):=\frac12  m(dx) =\begin{cases}
2\beta dx, &x>0,\\
2(1-\beta)dx, &x<0.
\end{cases}
\end{equation}
Notice that for small values on $\varepsilon$
\begin{equation}
\label{mstar}
m^*((x-\varepsilon, x+\varepsilon))=\begin{cases}
4\beta \varepsilon, &x>0,\\
2\varepsilon, & x=0,\\
4(1-\beta)\varepsilon, &x<0.
\end{cases}
\end{equation}

The basic result in the next theorem, which says that the local time $\ell_\beta$ is discontinuous at 0, was proved in \cite{walsh78a} using the It\^o-Tanaka formula. In \cite{walsh78a} it is pointed out that a proof could have also been done more directly based the occupation time formula and elementary calculus. Our proof is perhaps more on these lines.  Notice, however,  that the speed measure we are using is not the same as in \cite{walsh78a}, cf.  Remark (\ref{alpha}) below. 

\begin{theorem}
\label{prop10}
Let $L_\be$ and $\ell_\be$ be as above. Then 
\begin{description}
\item{a)} $(t,x)\mapsto   L_\be(t,x)$ is continuous a.s.,
\item{b)}  $x\mapsto    \ell_\be(t,x)$ is for every $t>0$  continuous a.s for $x\not=0$ but discontinuous a.s. at $0,$
\item{c)}  
\begin{equation}
\label{local}
 \ell_\be(t,x)= \begin{cases}
2\beta\,   L_\be(t,x), &x>0,\\
 L_\be(t,0), &x=0,\\
2(1-\beta)\,   L_\be(t,x), &x<0,
\end{cases}
\end{equation}
and
\begin{align}
\label{ee11}\ell_\be(t,0+)-\ell_\be(t,0-)&=2(2\be-1)L_\be(t,0),\\
\label{ee12}\ell_\be(t,0+)+\ell_\be(t,0-)&=2L_\be(t,0)=2\ell_\be(t,0).
\end{align} 
\end{description}

\end{theorem}

\noindent
The proof of Proposition \ref{prop10} is based on the random time change-techniques. For this approach we study first SBM in natural scale, i.e., the diffusion $Y=(Y(t))_{t\geq 0} := (S(W_\beta(t))_{t\geq 0}.$    

\begin{proposition} \label{prop11}
The speed measure $m_Y$ of SBM in natural scale is 
\begin{equation}
\label{speed}
m_Y(dx):=\begin{cases}
2\,(2\beta)^2 dx, &x>0,\\
2\,(2(1-\beta))^2 dx, &x<0.
\end{cases}
\end{equation}

\end{proposition}
\begin{proof} Assume first that $W_\beta(0)=x>0$ and define for $0<a<x<b$ 
$$
H_ {a,b}:=\inf\{t\,:\, W_\beta(t)\not\in(a,b)\}.
$$ 
Since $S(y)= y/{2\be}$ for $y>0$ we have by  It\^o's formula
\begin{align*} 
Y(t\wedge H_{a,b})-Y(0)&=S(W_\beta(t\wedge H_{a,b}))-S(x)\\
&=\int_0^{t\wedge H_{a,b}}S'(W_\beta(s))\, dW_\beta(s)\\
&=\int_0^{t\wedge H_{a,b}}\frac 1{2\beta} \, dW(s),
\end{align*}
from which we may read off the expression for the speed measure on $(0,\infty).$ Similar calculations reveal the claimed formula  on $(-\infty,0).$ We remark also that it holds $m_Y(\{0\})=0$ since $Y$ does not -- by construction -- have any sticky points.  
\end{proof}

\begin{proposition} \label{prop12}
Let $(\widehat W(t))_{t\geq 0}$ and $( W_\beta(t))_{t\geq 0}$ be a standard  and a skew Brownian motion, respectively, both   starting from 0.  Introduce
\begin{equation*}
\label{speed2}
\widehat m_Y(x):=(2\be)^2{\bf 1}_{[0,+\infty)}(x)+(2(1-\be))^2{\bf 1}_{(-\infty,0)}(x),
\end{equation*}
the additive functional  
$$
A_t:=\int_0^t \widehat m_Y(\widehat W(s))\, ds
$$
and its inverse
$$
\gamma_t:=\inf\{s\,:\, A_s>t\}.
$$
Then 
\begin{equation}
\label{eqlaw}
W_\beta(t)\stackrel{d}{=} S^{-1}( \widehat W(\gamma_t)) .
\end{equation}
where $\stackrel{d}{=} $ means that the processes on the left and the right hand side are identical in law.
\end{proposition}
\begin{proof} Since $\gamma_t$ is for all $t\geq 0$ a finite stopping time the process 
$$
Z:=(Z(t))_{t\geq 0}:=(\widehat W(\gamma_t))_{t\geq 0} 
$$
is a local  continuous $\cF_{\gamma_t}$-martingale with the quadratic variation process $(\gamma_t)_{t\geq 0}.$ Notice that 
$$
A'_t:=\frac d{dt} A_t= \widehat m_Y(\widehat W(t))\ {\rm a.s.},
$$
where we used that a.s. $W(t)\not= 0. $ Because $\gamma_t$ is the inverse of $A_t$, it holds
$$
\gamma'_t= \frac 1{\widehat m_Y(\widehat W(\gamma_t))}= \frac 1{\widehat m_Y(Z(t))}.
$$
Consequently, by the martingale representation theorem we may write for a Brownian motion $W$ (recall also that  $Z(0)=0)$ 
\begin{align*}
Z(t)&=\int_0^t \Big(\gamma'_s\Big)^{1/2} \,dW(s)\\
&=\int_0^t \Big(\frac 1{2\beta}{\bf 1}_{\{Z(s)>0\}} +\frac 1{2(1-\beta)}{\bf 1}_{\{Z(s)<0\}}\Big)\,dW(s), 
\end{align*}
and $Z$ can be seen as a weak solution of a SDE, and, therefore, a diffusion with the speed measure  as given in (\ref{speed}).
Hence, $Z$ is identical in law with a SBM in natural scale, and (\ref{eqlaw}) follows. 
\end{proof}

\noindent 
{\sl Proof of Theorem \ref{prop10}. } Let $\widetilde L_\be(t,x),\, t\geq 0, x\in\R,$ denote the local time of $(S^{-1}( \widehat W(\gamma_t)))_{t\geq 0}$ with respect to the measure $m^*,$ i.e.,  
$$
\widetilde L_\be(t,x):=\lim_{\varepsilon\downarrow 0} \frac{1}{m^*((x-\varepsilon, x+\varepsilon))}   \int_0^t {\bf 1}_{(x-\varepsilon, x+\varepsilon)}
 (S^{-1}(\widehat W(\gamma_s))\,ds.
 $$
Then from (\ref{eqlaw}) follows that 
 $\widetilde L_\be$ and $L_\be$ are identical in law as processes in $(t,x).$ We consider now  $\widetilde L_\be$. 
 For $x>\varepsilon>0$ it holds 
\begin{align*}
\{x-\varepsilon < S^{-1}(\widehat W(\gamma_s))<x+\varepsilon\}\ & \Leftrightarrow\ 
\{S(x-\varepsilon) < \widehat W(\gamma_s)<S(x+\varepsilon)\}\\
& \Leftrightarrow\   \Big\{\frac{x-\varepsilon}{2\beta} < \widehat W(\gamma_s)<\frac{x+\varepsilon}{2\beta}\Big\},
\end{align*}
and, hence, letting $O:={\displaystyle \Big(\frac{x-\varepsilon}{2\beta}, \frac{x+\varepsilon}{2\beta }\Big)}$
\begin{align*}
 \int_0^t {\bf 1}_{(x-\varepsilon, x+\varepsilon)}
 (S^{-1}(\widehat W(\gamma_s))\,ds 
 &= \int_0^t {\bf 1}_O(\widehat W(\gamma_s))\,ds\\
&= \int_0^{\gamma_t} \widehat m(\widehat W(u)) {\bf 1}_O(\widehat W(u))\,du\\
&= 4\beta^2\,\int_0^{\gamma_t}  {\bf 1}_O(\widehat W(u))\,du,
\end{align*}
where we have substituted $ s=A_u.$ Consequently, using (\ref{mstar}),
\begin{align*}
\frac{1}{m^*((x-\varepsilon, x+\varepsilon))}  \int_0^t {\bf 1}_{(x-\varepsilon, x+\varepsilon)}
& (S^{-1}(\widehat W(\gamma_s))\,ds\\
&=\frac 1{4\beta\varepsilon}\,4\beta^2\,\int_0^{\gamma_t}  {\bf 1}_O(\widehat W(u))\,du\\
&=\frac 1{|O|}\,\int_0^{\gamma_t}  {\bf 1}_O(\widehat W(u))\,du,
\end{align*}
where $|O|=\varepsilon/\beta$ is the length of $O.$ Letting here $\varepsilon\downarrow 0$ yields 
\begin{align*}
\widetilde L_\beta(t,x)&{=}\lim_{\varepsilon\downarrow 0} \frac 1{|O|}\,\int_0^{\gamma_t}  {\bf 1}_O(\widehat W(u))\,du
=\ell^{BM}\Big(\gamma_t, \frac x{2\beta}\Big),
\end{align*}
 where $\ell^{BM}$ denotes the local time of a standard BM with respect to the Lebesgue measure. Analogously, it can be proved that for $x<0$
$$  
\widetilde L_\beta(t,x){=}\ell^{BM}\Big(\gamma_t, \frac x{2(1-\beta)}\Big), 
 $$
 and for $x=0$
$$  
\widetilde L_\beta(t,x){=}\ell^{BM}(\gamma_t, 0). 
 $$ 
 Since $(t,x)\mapsto \ell^{BM}(t, x)$ is (jointly) continuous a.s. and $\widetilde L_\be$ is expressed as above in terms of $\ell^{BM}$
  it follows that  $(t,x)\mapsto \widetilde L_\beta(t, x)$  and, hence, also $(t,x)\mapsto   L_\beta(t, x)$ are continuous a.s. This proves a).  Clearly, (\ref{local}) follows directly from (\ref{mstar}) which when combined with a) yields b).   \hskip2.2cm $\square$ 

\begin{remark} 
\label{alpha}
Notice that  we could alternatively use, e.g.,  
$$
m_\alpha(dx) =\begin{cases}
\frac {1}{1-\alpha}  dx, &x>0,\\
\frac {1}{\alpha} dx, &x<0,
\end{cases}
$$
and 
$$
S_\alpha (x) =\begin{cases}
2(1-\alpha)x, &x\geq0,\\
2\alpha x, &x\leq0,
\end{cases}
$$
as the speed measure and the scale function, respectively. Also here $\alpha\in (0,1).$ The generator with this normalization  is as in (\ref{e0}) and (\ref{e1}) -- one simply replaces $\beta$ with $\alpha.$ The  $\alpha$-normalization is used, e.g., in Walsh \cite{walsh78a} and  Harrison and Shepp \cite{harrisonshepp81},  and the $\beta$-normalization, e.g., in    \cite{itomckean74} p. 115 and \cite{borodinsalminen15}. In this paper we work with the $\beta$-normalization. In particular,  the measure $m_\beta $ is used when the local time is considered with respect to $\frac 12\times$the speed measure. 
\end{remark}

 In the remark to follow we explicitly display for a fairly general diffusion the random time change, the scale change and the normalization of the local time such that the local time has a clean representation in terms of the Brownian local time. Such a representation is well-known  in the literature on diffusions, see  \cite{itomckean74} p. 174 and \cite{rogerswilliams87} (49.1) Theorem p. 289. However, we find it worthwhile to formulate the result in case the diffusion is not in natural scale.   
  
\begin{remark} Let $X=(X(t))_{t\geq 0}$ be a diffusion on $\R$ satisfying the SDE
\begin{equation}
\label{sde1}
dX(t)=a(X(t))\,dW(t)+ b(X(t))\,dt,
\end{equation} 
where coefficients $a$ and $b$ are nice functions, e.g., continuous, bounded, and $a(x)>\varepsilon>0$ for all $x\in \R,$    so that (\ref{sde1}) has a unique (non exploding) weak solution. Then the scale function $S$ is in ${\bf C}^2,$ and the speed measure  $m_Y$ of $Y:=(S(X(t)))_{t\geq 0}$ can be calculated using It\^o's formula (cf. the proof of  Proposition  \ref{prop11}) to be for $y\in S(\R):=\{S(x)\,:\, x\in\R\}$
$$
m_Y(dy)=2 \big(S'(S^{-1}(y))a(S^{-1}(y))\big)^{-2}\,dy =:2\widehat m_Y(y)\,dy.
$$
Similarly as in the case of SBM  (cf. the proof of Proposition \ref{prop12})  $Y$ can be constructed via a random time change based on the additive functional
$$
A_t:=\int_0^t \widehat m_Y(\widehat W(s))\, ds ,
$$
where $(\widehat W(t))_{t\geq 0}$ is a Brownian motion. Letting $(\gamma_t)_{t\geq 0}$ be the inverse of $A$ we have (cf. the proof of Proposition \ref{prop12})  
\begin{equation}
\label{eqlaw1}
 X(t)\stackrel{d}{=}S^{-1}(\widehat W(\gamma_t)). 
\end{equation}   
Recall (see \cite{borodinsalminen15} p. 17) that the speed measure of $X$ is given by
$$
m(dx) =2\,a(x)^{-2}\, S'(x)^{-1}\, dx =:2\,m^*(x)\,dx.
$$
Now  the local time of $X$ with respect to $m^*(dx):=m^*(x)\,dx$ can be expressed in terms of the Brownian local time as 
\begin{align*}
L&(t,x):=\lim_{\varepsilon\downarrow 0} \frac{1}{m^*((x-\varepsilon, x+\varepsilon))} \int_0^t {\bf 1}_{(x-\varepsilon, x+\varepsilon)}(X(s))\,ds \\
&\stackrel{d}{=}\lim_{\varepsilon\downarrow 0} \frac 1{{m^*((x-\varepsilon, x+\varepsilon))}}\,\int_0^{\gamma_t} \widehat m_Y(\widehat W(u)) {\bf 1}_{(S(x-\varepsilon),S(x+\varepsilon)}(\widehat W(u))\,du\\
&=\frac {\widehat m_Y(x)\, S'(x)}{m^*(x)} \lim_{\varepsilon\downarrow 0}\frac {1}{S(x+\varepsilon)-S(x-\varepsilon)} \,\int_0^{\gamma_t}  {\bf 1}_{(S(x-\varepsilon),S(x+\varepsilon)}(\widehat W(u))\,du\\
&=\ell^{BM}\Big(\gamma_t, S(x)\Big).
\end{align*}
 \end{remark}
 


\section{Markov property}
\label{s:2}
For the rest of the paper, let  $\bP_x$ and $\bE_x$ denote the probability and the expectation
associated with  SBM initiated at $x$. The theorems in this section characterize  Ray-Knight type results for $W_\be.$  Recall that $\ell_\be$ and $L_\be$  denote the local time with respect to the Lebesgue measure and the measure $m^*$ as given in (\ref{mstar0}).


\begin{theorem}\label{th 2.1} Let $\tau$ be an exponentially distributed random variable with parameter $\lambda>0.$  Given that  $W_\be(0)=0$ and $W_\be(\tau)=z>0$ the process
$(\ell_\be(\tau,y))_{y\in\bR}$ is a strong Markov process which can be represented  in the form
$$
\ell_\be(\tau,y)=\begin{cases}
V_1(y-z) &\text{for } \ z\le y,\\
V_2(y) &\text{for } \ 0<y\le z,\\
V_3(-y) &\text{for } \ y<0,
\end{cases}
$$
where $V_k(h)$, $h\ge 0$, $k=1,2,3$, are homogeneous diffusions with  the generating operators 
$$
\BL_1=2v\Big(\bfrac{d^2}
{d v^2}-\sqrt{2\la}\bfrac{d}{d v}\Big),\quad
\BL_2=2v\Big(\bfrac{{d}^2}
{d v^2}-\sqrt{2\la}\bfrac{d}{d v}\Big)
+2\bfrac{d}{d v},
$$
$$
\BL_3=2v\Big(\bfrac{{d}^2}
{d v^2}-\sqrt{2\la}\bfrac{d}{d v}\Big)$$
respectively. The initial values given that $W_\be(\tau)=z$ satisfy the equalities 
\begin{align}
\nonumber
V_1(0)&=V_2(z),\\
\label{121}
V_2(0)&=\ell_\be(\tau,0+)
{=}2\be\,\ell_\be(\tau,0),\\
\label{122}
V_3(0)&=\ell_\be(\tau,0-)
{=}2(1-\be)\,\ell_\be(\tau,0).
\end{align}
The conditional distribution of  $\ell_\be(\tau,0)$ is exponential with parameter $\sqrt{2\lambda}$: 
\begin{equation}
\label{2.2}
\bP_0(\ell_\be(\tau,0)\geq v|W_\be(\tau)=z)
={\rm e}^{-v\sqrt{2\la}}.
\end{equation}
The processes $V_k,$ $k=1,2,3,$ are conditionally independent given their initial values. 
\end{theorem}
\noindent
The proof is given in the appendix. It is  based on the Feynman-Kac formula and follows the structure of a proof of the corresponding theorem for standard BM as given in \cite{borodin13}. We remark that for $ z<0 $ an analogous description is valid and can be deduced from the fact that  if $W_\be(0)=0$ then
\begin{equation}
\label{as}
W_\be(t)\stackrel{d}=-W_{1-\be}(t).
\end{equation}

\begin{remark}\label{re 2.1} {\bf (i)} Notice that the generators given above are the same as in the Ray-Knight theorem of standard BM stopped at $\tau$  (see $\cite{borodinsalminen15}$ p. 91). Moreover, the distribution in (\ref{2.2}) does not depend on $z$ or $\beta.$  These facts are intuitively plausible from the excursion description of SBM.

\noindent
{\bf (ii)} Formulae (\ref{121}) and (\ref{122}) can also be derived from (\ref{ee11}) and (\ref{ee12}).
\end{remark}

Next we describe the law of $(L_\be(\tau,y))_{y\in\bR}$ that is,  the law of the continuous version of the local time conditioned on $W_\be(\tau)=z.$ We consider only the case $z>0,$ and leave the other case to the reader. The results are obtained straightforwardly from the connection between $\ell_\be$ and $L_\be$ as displayed in (\ref{local}) using Theorem \ref{th 2.1}.

\begin{theorem}\label{new11} Given $W_\be(\tau)=z>0$ the process
$(L_\be(\tau,y))_{y\in\bR}$  is a continuous strong Markov process which can be represented  in the form
$$
L_\be(\tau,y)=\begin{cases}
U_1(y-z) &\text{for } \ z\le y,\\
U_2(y) &\text{for } \ 0\leq y\le z,\\
U_3(-y) &\text{for } \ y\leq 0,
\end{cases}
$$
where $U_k(h)$, $h\ge 0$, $k=1,2,3$, are homogeneous diffusions having generating operators
$$\widetilde\BL_1=2v\Big(\bfrac 1{2\beta}\bfrac{d^2}
{d v^2}-\sqrt{2\la}\bfrac{d}{d v}\Big),\quad
\widetilde\BL_2=2v\Big(\bfrac 1{2\beta} \bfrac{{d}^2}
{d v^2}-\sqrt{2\la}\bfrac{d}{d v}\Big)
+\bfrac 1{\beta}\bfrac{d}{d v},
$$
$$
\widetilde\BL_3=2v\Big(\bfrac 1{2(1-\beta)}\bfrac{{d}^2}
{d v^2}-\sqrt{2\la}\bfrac{d}{d v}\Big)$$
respectively. The initial values satisfy  $U_1(0)=U_2(z)$ and $U_2(0)=U_3(0),$ where 
the conditional distribution of $U_2(0)=\ell_\be(\tau,0)=L_\be(\tau,0)$ given $W_\be(\tau)=z$  is as in (\ref{2.2}). The processes $U_k,$ $k=1,2,3,$ are conditionally independent given their initial values. 
\end{theorem}

  \section{Functionals of the local time process}
\medskip

Consider, for a fixed $t>0,$  the local time process $(\ell_\be(t,y))_{y\in\R}$ of SBM. In this section we study the integral functionals of the form
\begin{equation} 
B(t):=\int_{-\infty}^{\infty} f(\ell_\be(t,y))\,dy,\label{3.1}
\end{equation} 
where $f(v)$, $v\in [0,\infty)$, is a non-negative piecewise  
 continuous function. Let $\lambda >0$ and $\tau$  an exponentially with mean $1/\lambda$ distributed random variable independent of $W_\be.$  The main result characterizes the distribution of 
 \begin{equation} 
B(\tau):=\int_{-\infty}^{\infty} f(\ell_\be(\tau,y))\,dy,\label{3.11}
\end{equation}
via its Laplace transform  which can be obtained by solving a system of ordinary differential equations. To find the Laplace transform of $B(t)$ one should then invert the transform of $B(\tau)$  with respect to $\lambda.$ In the theorem to follow we, in fact, analyze  $B(\tau)$ under a restriction that the local time process does not exceed a given positive value. This allows us to compute the distribution of the supremum of the local time process.


\begin{theorem}\label{th 3.1} Let $f(v)$, $v\in [0,h]$, be a nonnegative
piecewise continuous function, $f(0)=0$, and $\be^*=\max\{\be,(1-\be)\}$. Then
\begin{equation} 
\bE_0\Big [\exp\Big(-\int_{-\infty}^{\infty} f(\ell_\be(\tau,y))\,dy\Big);
\sup_{y\in\bR}\ell_\be(\tau,y)\le h\Big]\hskip20mm
\label{3.3}
\end{equation}
$$
=2\la\int_0^{h/2\be^*}\big\{\be R(2(1-\be)v) Q(2\be\,v)
+(1-\be) R(2\be\,v) Q(2(1-\be) v)\big\}\,dv,
$$ 
where for $v\in[0,h]$ the functions $R$, $Q$ are the unique bounded
continuous solutions of the problem
\begin{align} 
& 2vR''(v)-(\la v+f(v)) R(v)=0,\qquad R(0)=1,\label{3.4}\\
&2vQ''(v)+2Q'(v)-(\la v+f(v)) Q(v)=-R(v),\label{3.5}\\
&R(h)=0,\qquad Q(h)=0.\label{3.6}
\end{align}
 In the case $h=\infty$ the boundary conditions in 
$(\ref{3.6})$ are replaced by
\begin{equation} 
\limsup_{v\to\infty} {\rm e}^{v\sqrt{\la/2}} R(v) <\infty,\qquad
\limsup_{v\to\infty} {\rm e}^{v\sqrt{\la/2}} Q(v) <\infty.\label{3.7}
\end{equation}
\end{theorem}
\noindent
{\sl  Proof.} Note that we consider only the case $W_\be(0)=0$.
Assume first that $h=\infty$.
Using (\ref{2.2}) and the Markov property
of the process $\ell_\be(\tau,y)$, see (\ref{2.3}) for $q=0$, yield
$$
\bE_0\Big\{\exp\Big(-\int_{-\infty}^\infty f(\ell_\be(\tau,y))\,dy\Big)\Big|
W_\be(\tau)=z\Big\}\hskip40mm
$$
$$
=\sqrt{2\la}\int_0^{\infty} {\rm e}^{-v\sqrt{2\la}}\,
\bE_0^z\Big\{\exp\Big(-\int_{-\infty}^\infty f(\ell_\be(\tau,y))\,dy\Big)\Big|
\ell_\be(\tau,0)=v\Big\}\,dv
$$
\begin{equation}
=\sqrt{2\la}\int_0^{\infty} {\rm e}^{-v\sqrt{2\la}}\,
\wt r(z,v)\wt q(z,v)\,dv,\label{3.8}
\end{equation}
where
\begin{align}\wt r(z,v):&=\bE_0^z\Big\{\exp\Big(-\int_{-\infty}^0 f(\ell_\be(\tau,y))\,dy\Big)
\Big|\ell_\be(\tau,0)=v\Big\},\nonumber\\
\wt q(z,v):&=\bE_0^z\Big\{\exp\Big(-\int_0^\infty
f(\ell_\be(\tau,y))\,dy\Big)\Big|\ell_\be(\tau,0)=v\Big\},\nonumber\end{align}
and  $\bE^z_0$ stands for the expectation associated with the conditional measure 
$\bP_0^z(\cdot):=\bP_0\{\cdot|W_\be(\tau)=z\}.$ We assume now that  $z>0$. Applying (\ref{122}) and the definition of  $V_3$,
it is seen that 
$$
\wt r(z,v)=\bar r(z,2(1-\be)v),
$$ where
$$\bar r(z,v):=\bE\Big\{\exp\Big(-\int_0^\infty f(V_3(h))\,dh\Big)\Big|V_3(0)=v\Big\}. $$

\noindent
Since the process $V_3$ has
the generating operator 
$$
\BL_3=2v\Big(\bfrac{{d}^2}
{d v^2}-\sqrt{2\la}\bfrac{d}{d v}\Big)
$$
(cf. Theorem \ref{th 2.1}) it follows  that $\bar r(z,v)$ does not depend on $z$.
Let  

\begin{equation}\label{barR}
\bar R(v):=\bar r(z,v).
\end{equation} 
Then by the Feynman--Kac formula (cf.  Theorem~12.5 Ch.~{II} of \cite{borodin13}) we have that the function $\bar R$ is a bounded solution of the ODE
\begin{equation}
2v\big(\bar R''(v)-\sqrt{2\la} \,\bar R'(v)\big)-f(v)\bar R(v)=0,\quad v>0.\label{3.9}
\end{equation}
The operator $\BL_3$ corresponds to the $0$-dimensional squared radial Ornstein-Uhlenbeck process. 
The boundary point 0 is taken to be killing and, hence, when hitting zero the process 
never leaves zero. Since
$f(0)=0$ it follows that $\bar R(0)=1.$ We remark also that $V_1$ behaves similarly as $V_3.$ 

Applying (\ref{121}) and the definitions of $V_1$ and  $V_2$ yield 
$$\wt q(z,v)=\bar q(z,2\be\,v),$$ where
\begin{align*}
\bar q(z,v)&:=\bE_v^{(2)}\Big\{\exp\Big(-\int_0^\infty f(V_1(h))\,dh
-\int_0^z f(V_2(h))\,dh\Big)\Big\}\\
&:=\bE\Big\{\exp\Big(-\int_0^\infty f(V_1(h))\,dh
-\int_0^z f(V_2(h))\,dh\Big)\Big|V_2(0)=v\Big\} \\ 
&
=\int_0^\infty\bE_v^{(2)}\Big\{\!\exp\!\Big(\!-\int_0^\infty f(V_1(h))\,dh
-\int_0^z f(V_2(h))\,dh\!\Big)\Big|V_2(z)=g\Big\}
\\
&\hskip3cm \times 
\bP\big(V_2(z)\in dg\,|\, V_2(0)=v\big). 
\end{align*}
By the Markov property and
the condition $V_1(0)=V_2(z)$ we have
\begin{align}
\bar q(z,v) &=\int_0^\infty\bE\Big\{\exp\Big(-\int_0^\infty f(V_1(h))\,dh\Big)\Big|V_1(0)=g\Big\}\nonumber\\
&\times\bE_v^{(2)}\Big\{\exp\Big(-\int_0^zf(V_2(h))\,dh\Big)\Big|V_2(z)=g\Big\}
\bP_v^{(2)}\big(V_2(z)\in dg\big),\nonumber
\end{align}
where $\bP_v^{(2)}(\cdot):=\bP(\cdot\,|\, V_2(0)=v).$  Since $V_1$ and $V_3$ (with the same initial values) are identical in law we may write (cf. (\ref{barR}))
\begin{align}
\nonumber
\bar q(z,v) &=\int_0^\infty\bar R(g)
\bE_v^{(2)}\Big\{\exp\Big(-\int_0^zf(V_2(h))\,dh\Big)\Big|V_2(z)=g\Big\}\\
&
\nonumber \hskip2cm \times
\bP_v^{(2)}\big(V_2(z)\in dg\big)\\
&=\bE\Big\{\bar R(V_2(z))\exp\Big(-\int_0^zf(V_2(h))\,dh\Big)\Big|V_2(0)=v\Big\},\nonumber
\end{align}
and applying the Feynman--Kac formula for diffusion $V_2$ yields that $\bar q$ satisfies
\begin{align}\nonumber\hskip-1cm\bfrac{\partial}{\partial z}\bar q(z,v)&=
2v\Big(\bfrac{\partial^2}{\partial v^2}\bar q(z,v)-\sqrt{2\la}
\bfrac{\partial}{\partial v}\bar q(z,v)\Big)
\\
&\hskip2cm +
2\bfrac{\partial}{\partial v}\bar q(z,v)-f(v)\bar q(z,v),\label{3.10}\\
\bar q(0,v)&=\bar R(v).\label{3.11}
\end{align}
Now we re-express ( \ref{3.8})  as
$$
\bE_0\Big\{\exp\Big(-\int_{-\infty}^\infty f(\ell_\be(\tau,y))\,dy\Big)\Big|
W_\be(\tau)=z\Big\}\hskip40mm
$$
$$
=\sqrt{2\la}\int_0^{\infty} {\rm e}^{-v\sqrt{2\la}}\,
\bar R(2(1-\be)v)\bar q(z,2\be\,v)\,dv, 
$$
\begin{equation}
\hskip-1.1cm=\sqrt{2\la}\int_0^{\infty} R(2(1-\be)v) q(z,2\be\,v)\,dv,\label{3.12}
\end{equation}
where $R(v):={\rm e}^{-v\sqrt{\la/2}}\bar R(v)$ and $q(z,v):={\rm e}^{-v\sqrt{\la/2}}\bar q(z,v)$.
Since $\bar R$ satisfies in (\ref{3.9}) it is seen that $R$ satisfies (\ref{3.4}). Furthermore,   since $\bar q$ 
fulfills (\ref{3.10}) and (\ref{3.11}) we have 
\begin{align}\nonumber\hskip-1cm
\bfrac{\partial}{\partial z} q(z,v)&=
2v\bfrac{\partial^2}{\partial v^2} q(z,v)+2\bfrac{\partial}{\partial v} q(z,v)\\
&\hskip2cm -
(\la v-\sqrt{2\la}+f(v)) q(z,v),\label{3.13}\\
 q(0,v)&=R(v).\label{3.14}
\end{align}
Observing that  $\bar R$ is bounded the first inequality in (\ref{3.7}) clearly holds.
Define
$$ Q(v):=\int_0^{\infty} {\rm e}^{-z\sqrt{2\la}} q(z,v)\,dz. $$
Then  ${\rm e}^{v\sqrt{\la/2}}\,Q(v)$, $v\ge 0$, is bounded because $\bar q$ is bounded.  Therefore the second inequality  in (\ref{3.7}) holds. From (\ref{3.13}) and (\ref{3.14}) it follows that
$Q$ satisfies (\ref{3.5}).
Hence, taking into account (\ref{1.10}), we have
\begin{align*}
\bE&_0\Big\{\exp\Big(-\int_{-\infty}^{\infty} f(\ell_\be(\tau,y))\,dy\Big)
{\bold 1}_{(0,\infty)}(W_\be(\tau))\Big\}\\
&=\sqrt{2\la}\be\int_0^\infty dz {\rm e}^{-z\sqrt{2\la}}
\bE_0\Big\{\exp\Big(-\int_{-\infty}^\infty f(\ell_\be(\tau,y))\,dy\Big)\Big|
W_\be(\tau)=z\Big\}
\\
&=2\la\be\int_0^\infty dv\int_0^\infty dz {\rm e}^{-z\sqrt{2\la}} R(2(1-\be)v)
q(z,2\be v)
\\
&=2\la\be\int_0^\infty dv  R(2(1-\be)v)\,Q(2\be v).
\end{align*}
Using (\ref{as}) we
obtain (\ref{3.3}) for the case when $ h=\infty$ and $f$ is  bounded and
twice continuously differentiable with bounded first and
second derivatives. 
The extension  for the piecewise continuous
functions is proved using  approximations by continuously
differentiable functions (see  the proof of Theorem 4.1 Ch.~{IV} in \cite{borodin13}). 

The case $h<\infty$  can be analyzed similarly  as
in the proof of Theorem 5.1 Ch.~{V} in \cite{borodin13}. This completes the proof of Theorem \ref{th 3.1}.\hfill $\square $

\begin{remark}\label{re 3.2} In the case $h<\infty$ we indeed consider the solutions
of equations $(\ref{3.4})$--$(\ref{3.6})$ to be equal to zero for $v\in[h,\infty)$. 
Then in $(\ref{3.3})$ one can integrate from 0 up to $\infty$.
\end{remark}

Theorem \ref{th 3.1} can be applied to compute the distribution of the supremum of the 
local time process  $(\ell_\be(t,y))_{y\in\bR}$ up to a fixed time $t>0$. In the rest of this section we write, for short,
$$
\sup\ell_\be(t,y):= \sup\{\ell_\be(t,y)\,:\, y\in\bR\}.
$$
The distribution of the supremum is first found  Theorem  \ref{th 3.2} up to an exponential time $\tau.$  
The inversion is then given in Theorem \ref{th 3.3}
\begin{theorem}\label{th 3.2} For $h\ge 0$ and $\be\ge 1/2$
\begin{equation}
\bP_0\Big(\sup\ell_\be(\tau,y)>h\Big)
=\bfrac{h\sqrt{2\la}I_1(\be,h\sqrt{\la/2})}
{\sh^2(h\sqrt{\la/2})I_0(h\sqrt{\la/2})}, \label{3.15}
\end{equation}
where with $\widehat\beta=(1-\beta)/\beta$
$$I_1(\be,x):=\bfrac12\int_0^1\big\{\sh\big((1-\widehat\be v)x\big)
I_0(vx)+\widehat\be\sh((1-v)x) I_0\big(\widehat\be vx\big)\big\} dv,$$
and $I_0$ and $I_1$ denote the modified Bessel functions of order 0 and 1, respectively.
\end{theorem}
\noindent
The corollaries below are special cases of (\ref{3.15}). The first one states the distribution of the supremum of local time of reflecting BM, which can also be found in \cite{borodinsalminen15} formula 3.1.11.2. The second one is for standard BM and coincides with formula 1.1.11.2 in ibid.

\begin{corollary}\label{co 3.1} For $\be=1$ we have $W_\be(t)=|W(t)|,$
\begin{equation}
I_1(1,x)=\bfrac12\,\sh( x)\int_0^1 I_0(xv)\,dv,\label{3.16}
\end{equation}
and
$$
\bP_0\Big(\sup\ell_1(\tau,y)>h\Big)=\bfrac{h\sqrt{\la/2}}
{\sh(h\sqrt{\la/2})I_0(h\sqrt{\la/2})} \int_0^1
I_0(vh\sqrt{\la/2})\,dv. $$

\end{corollary}

\begin{corollary}\label{co 3.2} For $\be=1/2$ we have $W_\be(t)=W(t),$
\begin{equation}
I(1/2,x):=\int_0^1 \sh((1-v)x) I_0(vx)\,dv=I_1(x), \label{3.17}
\end{equation}
and
\begin{equation}
\bP_0\Big(\sup\ell_{1/2}(\tau,y)>h\Big)=\bfrac{h\sqrt{2\la}\,I_1(h\sqrt{\la/2})}
{\sh^2(h\sqrt{\la/2})I_0(h\sqrt{\la/2})}. \label{3.18}
\end{equation}

\end{corollary}
\noindent

\noindent
{\bf Proof of Theorem \ref{th 3.2}.}  We use Theorem 3.1 with $f=0$. The solutions
of (\ref{3.4})--(\ref{3.6}) in this case are
$$ R(v)=\bfrac{\sh((h-v)\sqrt{\la/2})}{\sh(h\sqrt{\la/2})}, \qquad 0\le v\le h, $$
and
$$ Q(v)=\bfrac{\ch((h-v)\sqrt{\la/2})}{\sqrt{2\la}\sh(h\sqrt{\la/2})}
-\bfrac{I_0(v\sqrt{\la/2})}{\sqrt{2\la}\sh(h\sqrt{\la/2}) I_0(h\sqrt{\la/2})},
\qquad 0\le v\le h. $$
One can verify that
$$L_1:=\bfrac{\be\sqrt{2\la}}{\sh^2(h\sqrt{\la/2})}
 \int_0^{h/2\be}\sh((h-2(1-\be)v)\sqrt{\la/2})\,\ch((h-2\be v)\sqrt{\la/2})\,dv $$
$$=\bfrac{\be\sqrt{2\la}}{2\sh^2(h\sqrt{\la/2})}
 \int_0^{h/2\be}(\sh((h-v)\sqrt{2\la})+\sh((2\be-1)v\sqrt{2\la}))\,dv, $$
and
$$L_2:=\bfrac{(1-\be)\sqrt{2\la}}{\sh^2(h\sqrt{\la/2})}
 \int_0^{h/2\be}\sh((h-2\be v)\sqrt{\la/2})\,\ch((h-2(1-\be)v)\sqrt{\la/2})\,dv $$
$$=\bfrac{(1-\be)\sqrt{2\la}}{2\sh^2(h\sqrt{\la/2})}
 \int_0^{h/2\be}(\sh((h-v)\sqrt{2\la})-\sh((2\be-1)v\sqrt{2\la}))\,dv. $$
Then
\begin{align*}
L_1+L_2&=\bfrac{\sqrt{\la/2}}{\sh^2(h\sqrt{\la/2})}\int_0^{h/2\be}
\sh((h-v)\sqrt{2\la})\,dv\\
&\hskip2cm +\bfrac{\sqrt{\la/2}(2\be-1)}{\sh^2(h\sqrt{\la/2})}\int_0^{h/2\be}
\sh((2\be-1)v\sqrt{2\la})\,dv\\
&=\bfrac1{2\sh^2(h\sqrt{\la/2})}
\big((\ch(h\sqrt{\la/2})-\ch((1-\tfrac1{2\be})h\sqrt{\la/2}))\\
&\hskip2cm
+(\ch((1-\tfrac1{2\be})h\sqrt{\la/2})-1)\big)\\
&=1.
\end{align*}
Now applying (\ref{3.3}) with $\be\ge 1/2$ and $f\equiv 0$ we have
$$
\bP\Big(\sup\ell_\be(\tau,y)\le h\Big)=L_1+L_2\hskip60mm 
$$
$$-\bfrac{\sqrt{2\la}}{\sh^2(h\sqrt{\la/2})I_0(h\sqrt{\la/2})}
\int_0^{h/2\be}\big\{\be\sh((h-2(1-\be)v)\sqrt{\la/2})
I_0(2\be v\sqrt{\la/2})$$
$$+(1-\be)\sh((h-2\be v)\sqrt{\la/2})
I_0(2(1-\be)v\sqrt{\la/2})\big\} dv. $$
Changing variables in the integral yields (\ref{3.15}) and the theorem is proved.\hfill $\square $

\medskip

Next we focus on the inversion of the Laplace transform in (\ref{3.15}). Recall the following standard notations (see Appendix 2 in \cite{erdelyi54}): $J_0(x)$ and $J_1(x)$,
$x\in\bR,$ are Bessel functions and $0<j_1<j_2<\dots$ are the positive
zeros of $J_0(x)$. Also let with $\widehat\be=(1-\be)/\be$ 
$$I_0(\be,x):=\bfrac12\int_0^1\big\{\ch((1-\widehat\be v)x)
I_0(vx)+\widehat\be \ch((1-v)x) I_0(\widehat\be vx)\big\} dv, $$
$$
I_1(\be,x):=\bfrac12\int_0^1\big\{\sh((1-\widehat\be v)x)
I_0(vx)+\widehat\be \sh((1-v)x) I_0(\widehat\be vx)\big\} dv, 
$$
$$J_0(\be,x):=\bfrac12\int_0^1\big\{\cos((1-\widehat\be v)x)
J_0(vx)+\widehat\be \cos((1-v)x) J_0(\widehat\be vx)\big\} dv, $$
$$J_1(\be,x):=\bfrac12\int_0^1\big\{\sin((1-\widehat\be v)x)
J_0(vx)+\widehat\be \sin((1-v)x) J_0(\widehat\be vx)\big\} dv. $$
\noindent
Using the properties of Bessel functions we have
\begin{align}\label{3.19} 
I_1(\be,ix)&=i J_1(\be,x),\qquad\qquad\hskip.05cm I_0(\be,ix)=J_0(\be,x), \\
\nonumber J_1(1/2,x)&=J_1(x),\hskip1.86cm J_0(1/2,x)=J_0(x),\\
\nonumber J_1(1,x)&=\bfrac{\sin x}{2x} \int_0^x J_0(v)\,dv,\quad
J_0(1,x)=\bfrac{\cos x}{2x} \int_0^x J_0(v)\,dv. 
\end{align}
In the proof of the next result the derivative with respect to $x$ of $I_1(\be,x)$ is needed.  Integrating by parts this is obtained in the following form
\begin{equation}
\bfrac{d}{d x}I_1(\be,x)=\bfrac1{2x}\sh((2-\tfrac1\be)x)
-\bfrac1x I_1(\be,x)+I_0(\be,x).\label{3.20}
\end{equation}

\begin{theorem}\label{th 3.3} For $h\ge 0$ and $\be\ge 1/2$
$$
\bP\Big(\sup\ell_\be(t,y)\le h\Big)=4\sum_{k=1}^{\infty}
\bfrac{J_1(\be,j_k)}{\sin^2 j_k\,J_1(j_k)} {\rm e}^{-2 j_k^2 t/h^2}
+4\sum_{k=1}^{\infty}\Big[\bfrac{4 t\pi k J_1(\be,\pi k)}
{h^2 J_0(\pi k)}\hskip17mm$$
\begin{equation}
-\bfrac{\sin((2-\tfrac1\be)\pi k)}{\pi k}
-\bfrac{J_0(\be,\pi k)}{J_0(\pi k)}+\bfrac{J_1(\be,\pi k)}{\pi k J_0(\pi k)}
-\bfrac{J_1(\be,\pi k)J_1(\pi k)}{J_0^2(\pi k)}\Big]
{\rm e}^{-2\pi^2 k^2 t/h^2}.\label{3.21}
\end{equation}
\end{theorem}
\noindent
For reflecting BM and standard BM we have the following formulas which coincide with   3.1.11.4  and  1.1.11.4, respectively, in 
\cite{borodinsalminen15}. See also (5.22) in Ch. 5 in
$\cite{borodin13}$.
\begin{corollary}\label{co 3.3} For $\be=1$ we have $W_\be(t)=|W(t)|$,
$$J_1(1,\pi k)=0,\quad J_0(1,\pi k)
=\bfrac{(-1)^k}{2\pi k}\int_0^{\pi k} J_0(v)\,dv, $$
and
$$
\bP\Big(\sup\ell_1(t,y)\le h\Big)\hskip75mm $$
$$= \sum_{k=1}^\infty\Big\{\bfrac{2\,{\rm e}^{-2j_k^2t/h^2}}
 {j_k J_1(j_k)\sin j_k}\int_0^{j_k} J_0(v)dv
 -\bfrac{2\, {\rm e}^{-2\pi^2k^2t/h^2}}{(-1)^k\pi k J_0(\pi k)}\int_0^{\pi k} J_0(v)dv
 \Big\}.$$
\end{corollary}

\begin{corollary}\label{co 3.4} For $\be=1/2$ we have $W_\be(t)=W(t)$ and
$$
\bP\Big(\sup\ell_{1/2}(t,y)\le h\Big)= 4\sum_{k=1}^\infty
\bfrac1{\sin^2j_k} {\rm e}^{-2j_{0,k}^2t/h^2}\hskip20mm$$
$$+4\sum_{k=1}^\infty\Big(\bfrac{4t\pi kJ_1(\pi k)}{h^2J_0(\pi k)}+
 \bfrac{J_1(\pi k)}{\pi kJ_0(\pi k)}-
 \bfrac{J_1^2(\pi k)}{J_0^2(\pi k)}-1\Big) {\rm e}^{-2\pi^2k^2t/h^2}.$$

\end{corollary}

\noindent
{\bf Proof of Theorem \ref{th 3.3}.} Considering (\ref{3.15}) with $h=1$ and evoking  the formula for the inverse Laplace transform we have
$$
\bP_0\Big(\sup\ell_\be(t,y)>1\Big)
=\bfrac{1}{2\pi i}\lim_{\rho\rightarrow \infty}\int_{\gamma-i\rho}^{\gamma+i\rho}
{\rm e}^{\la t}\bfrac{\sqrt{2}I_1(\be,\sqrt{\la/2})}{\sqrt{\la}\,
\sh^2(\sqrt{\la/2}) I_0(\sqrt{\la/2})}\,d\la,$$
where $\gamma$ is some small positive constant. The integral can be computed using the residue theorem. For this, note first that  
$$\lim_{\la\rightarrow 0}\bfrac{\sqrt{2}}{\sqrt{\la}}I_1(\be,\sqrt{\la/2})
=\bfrac12\int_0^1\{(1-(\tfrac1\be-1)v)+(\tfrac1\be-1)(1-v)\} dv=\bfrac12.$$
Since $0<j_1<j_2<\dots$ are the positive zeros of the function
$J_0(x)$, the values $-2j_k^2$, $k=1,2,\dots$, are the zeros of the function $I_0(\sqrt{\la/2})$. The
residues $r_{1,k}(t)$ of the function
$$g(\la):={\rm e}^{\la t}\bfrac{\sqrt{2}I_1(\be,\sqrt{\la/2})}
{\sqrt{\la}\sh^2(\sqrt{\la/2}) I_0(\sqrt{\la/2})}$$
at the points $-2j_k^2$ are
$$r_{1,k}(t)={\rm e}^{-2j_k^2t}\bfrac{4 I_1(\be,ij_k)}{\sh^2(ij_k)I_0^{\prime}(ij_k)}
=-\bfrac{4 J_1(\be,j_k)}{\sin^2j_k J_1(j_k)}{\rm e}^{-2j_k^2t}.$$
Since $\la=0$ is the simple root of the function
$\sh^2(\sqrt{\la/2})$, the residue of $g(\la)$ at this point is
$$r_0(t)=\lim_{\la\rightarrow 0}\bfrac{\la {\rm e}^{\la t}}{2\sh^2(\sqrt{\la/2})}=1.$$
The points $\la_k=-2\pi^2k^2$, $k=1,2,\cdots$, are the double roots
of the denominator of the function $g(\la)$, so the residues at these
points are
\begin{align} r_{2,k}(t)&=\lim_{\la\rightarrow \la_k}\bfrac{d}{d\la}
\{(\la-\la_k)^2 g(\la)\}\nonumber\\
&=\bfrac1{((\sh(\sqrt{\la_k/2}))')^2}\bigg(
\bfrac{{\rm e}^{\la_k t}I_1(\be,\sqrt{\la_k/2})}{\sqrt{\la_k/2}I_0(\sqrt{\la_k/2})}\bigg)^\prime\\
&\hskip2cm -\bfrac{(\sh(\sqrt{\la_k/2}))''}{((\sh(\sqrt{\la_k/2}))')^3}\bfrac{{\rm e}^{\la_k t}I_1(\be,\sqrt{\la_k/2})}{\sqrt{\la_k/2}I_0(\sqrt{\la_k/2})}\nonumber\\
&=\Big[-16\bfrac{t\pi kJ_1(\be,\pi k)}{J_0(\pi k)}+4\bfrac{\sin((2-\tfrac1\be)\pi k)}{\pi k}
+4\bfrac{J_0(\be,\pi k)}{J_0(\pi k)}\nonumber\\
&\hskip2cm -4\bfrac{J_1(\be,\pi k)}{\pi kJ_0(\pi k)}
+4\bfrac{J_1(\be,\pi k)J_1(\pi k)}{J_0^2(\pi k)}\Big] {\rm e}^{-2\pi^2k^2t}.\nonumber\end{align}
Here we used the formulas (\ref{3.19}), (\ref{3.20}) and
$$(\sh(\sqrt{\la/2}))'\Big|_{\la=-2\pi^2k^2}=\bfrac{(-1)^k}{4i\pi k},\qquad
(\sh(\sqrt{\la/2}))''\Big|_{\la=-2\pi^2k^2}=\bfrac{(-1)^k}{16i\pi^3 k^3}.$$
Consequently, the inversion  based on the residue theorem results to 
$$
\bP\Big(\sup\ell_\be(t,y)>1\Big)=r_0(t)+\sum_{k=1}^\infty(r_{1,k}(t)
+r_{2,k}(t)).$$
Considering the probability of the complementary event and substituting the
expressions for the residues give
$$\bP\Big(\sup_{y\in\bR}\ell_\be(t,y)\le 1\Big)=4\sum_{k=1}^{\infty}
\bfrac{J_1(\be,j_k)}{\sin^2 j_k\,J_1(j_k)} {\rm e}^{-2 j_k^2 t}
+4\sum_{k=1}^{\infty}\Big[\bfrac{4 t\pi k J_1(\be,\pi k)}
{h^2 J_0(\pi k)}\hskip17mm$$
\begin{equation}
-\bfrac{\sin((2-\tfrac1\be)\pi k)}{\pi k}
-\bfrac{J_0(\be,\pi k)}{J_0(\pi k)}
+\bfrac{J_1(\be,\pi k)}{\pi k J_0(\pi k)}-
\bfrac{J_1(\be,\pi k)J_1(\pi k)}{J_0^2(\pi k)}\Big]
{\rm e}^{-2\pi^2 k^2 t}.\label{3.22}
\end{equation}
By the scaling property of SBM (see Appendix A.1 (\ref{1.11}))
\begin{equation}\label{scaling}
\bP\Big(\sup\ell_\be(t,y)\le h\Big)=
\bP\Big(\sup\ell_\be(t/h^2,y)<1\Big).
\end{equation}
Now substituting in (\ref{3.22}) the value $t/h^2$ in place of $t$, we get formula 
(\ref{3.21}) for the distribution of $\sup\ell_\be(t,y)$.
The theorem is proved.\hfill $\square $
\appendix
\section{Appendix}
\subsection{ Feynman-Kac formula}

\label{s:skew.br}

For bounded measurable $\Phi $ and non-negative measurable $g$ define
\begin{align}
\label{1.2} 
\nonumber
U(x):&=\la\int_0^\infty {\rm e}^{-\la t}\bE_x\Big\{\varPhi(W_\be(t))\exp\Big(-\int_0^t g(W_\be(s))\,ds\Big)\Big\}\,dt\\
&=\bE_x\Big\{\varPhi(W_\be(\tau))
\exp\Big(-\int_0^\tau g(W_\be(s))\,ds\Big)\Big\},
\end{align}
where $\tau$ is an exponentially with parameter $\lambda>0$ distributed random time independent of $W_\beta.$ 

\begin{theorem} \label{th 1.1} Let $\varPhi$ and $g$ be as above and also piecewise continuous. Then the function $U$ defined in $(\ref{1.2})$
is the unique bounded continuous solution of the differential problem
\begin{align} &\bfrac12 U''(x)-(\la+g(x))U(x)=-\la\varPhi(x),\qquad x\not=0,
\label{1.3} \\
& (1-\be)U'(0-)=\be U'(0+). \label{1.4} \end{align}
\end{theorem}

\begin{remark} \label{re 2.3} For piecewise continuous functions $\varPhi$ and $g$
equation $(\ref{1.3})$  holds at all points of
continuity of $\varPhi$ and $g$ but at the points of discontinuities the solution is constructed to be  continuously differentiable.
\end{remark}

Let for $\ga>0$ 
\begin{equation}\label{Gz}
G_z(x):=\bfrac{d}{d z}
\bE_x\Big\{\exp\Big(-\int_0^{\tau} f(W_\be(s))\,ds-\ga\ell_\be(\tau, q)\Big);
 W_\be(\tau)<z\Big\}. 
 \end{equation}

Then proceeding as Ch. III \S 4  in \cite{borodin13} we may deduce the following result.

\begin{theorem} \label{th 1.2} Let $f$ be nonnegative piecewise
continuous function. Then the function $ G_z$ for a fixed $z$ 
is the unique bounded continuous solution of the problem
\begin{align} &\bfrac12 G''(x)-(\la+f(x))G(x)=0,\qquad x\not=z,q,0,\label{2.6}\\
& G^{\prime}(z+0)-G^{\prime}(z-0)=-2\la,\label{2.7}\\
& G^{\prime}(q+0)-G^{\prime}(q-0)=2\ga G(q),\label{2.8}\\
& (1-\be)G^{\prime}(0-)=\be G^{\prime}(0+).\label{2.9}
\end{align}
\end{theorem}
\begin{remark} \label{re 2.31} 
Note that if $q=0$ then $(\ref{2.8})$ and  $(\ref{2.9})$ are combined to be 
\begin{equation}
\be G^{\prime}(0+)-(1-\be)G^{\prime}(0-)=\ga G(0), \label{2.10}
\end{equation}
and if q=z then $(\ref{2.7})$ and  $ (\ref{2.8})$ become
\begin{equation}
G^{\prime}(z+0)-G^{\prime}(z-0)=-2\la+2\ga G(z). \label{2.11}
\end{equation}
\end{remark}
\noindent
 Theorem \ref{th 1.2} can be used to calculate the Green kernel of skew BM; this is
 \begin{equation}
\bfrac{d}{dz}\bP_x(W_\be(\tau)<z)
=\bfrac{\sqrt\la}{\sqrt2}\,{\rm e}^{-|x-z|\sqrt{2\la}}
+\bfrac{\sqrt\la}{\sqrt2}\,(2\be-1)\,\sign z\,{\rm e}^{-(|x|+|z|)\sqrt{2\la}}.
\label{1.10}
\end{equation}
 
\noindent  
 Inverting (\ref{1.10}) yields the transition density  with respect to the Lebesgue measure  (cf. \cite{walsh78a}, and  AI, Nr. 12 \cite{borodinsalminen15})
 \begin{equation}\label{1.11}
\bfrac{d}{dz}\bP_x(W_\be(t)<z)
=\bfrac{{\rm e}^{-(x-z)^2/2t}}{\sqrt{2\pi t}}+(2\be-1)\sign z\,\bfrac{{\rm e}^{-(|x|+|z|)^2/2t}}
{\sqrt{2\pi t}}. 
\end{equation}
Notice that from (\ref{1.11}) it follows that SBM has the scaling property used in (\ref{scaling}).

\subsection{ Proof of Theorem \ref{th 2.1}}
\medskip\noindent
Let $f$ be an arbitrary continuous bounded
function and $q$  an arbitrary real number. Set
$$ f_+(x):=f(x){\bold 1}_{(q,\infty)}(x),\qquad f_-(x):=f(x)
{\bold 1}_{(-\infty,q]}(x). $$
Let $\cG_u^v=\si(\ell_\be(\tau,y),u\le y\le v)$ be the $\si$-algebra of events
generated by the local time $\ell_\be(\tau,y)$ on the interval $[u,v]$.

To prove that $\ell_\be(\tau,y)$, $y\in\bR$, is a Markov process
under the condition $W_\be(\tau)=z$ it suffices to verify that for 
any $q\in\bR$ and $v\in [0,\infty)$
\begin{align}
\bE_0^z &\Big\{\exp\Big(-\int_{-\infty}^{\infty} f(y)\ell_\be(\tau,y)\,dy\Big)\Big|
\ell_\be(\tau,q)=v\Big\}\nonumber\\
&=\bE_0^z\Big\{\exp\Big(-\int_{-\infty}^{\infty} f_+(y)
\ell_\be(\tau,y)\,dy\Big)\Big|\ell_\be(\tau,q)=v\Big\}\label{2.3}\\
&\times\bE_0^z\Big\{\exp\Big(-\int_{-\infty}^{\infty}
f_-(y)\ell_\be(\tau,y)\,dy\Big)\Big|\ell_\be(\tau,q)=v\Big\}.\nonumber
\end{align}
\noindent
Notice that by definition of the local time,
$$\int_{-\infty}^{\infty} f(y)\ell_\be(\tau,y)\,dy=\int_0^{\tau} f(W_\be(s))\,ds $$
and, consequently, the Markov property can be rewritten in the form
\begin{align}
\bE_0^z
\Big\{&\exp\Big(-\int_0^\tau f(W_\be(s))\,ds\Big)\Big|
\ell_\be(\tau, q)=v\Big\}\nonumber\\
&=\bE_0^z\Big\{\exp\Big(-\int_0^{\tau} f_+(W_\be(s))\,ds\Big)\Big|
\ell_\be(\tau, q)=v\Big\}\nonumber\\
&\times\bE_0^z\Big\{\exp\Big(-\int_0^{\tau} f_-(W_\be(s))\,ds\Big)\Big|
\ell_\be(\tau, q)=v\Big \}.\label{2.4}
\end{align}
\noindent 
We prove (\ref{2.4}) by computing for these expectations explicit formulas
expressed in terms of fundamental solutions of the equation
\begin{equation}
\bfrac12\phi^{\prime\prime}(y)-(\la+f(y))\phi(y)=0,\qquad y\in\bR.\label{2.5}
\end{equation}

\noindent 
Let $\psi$  and $\ph$ be the increasing and decreasing, respectively, positive solution
of equation (\ref{2.5}), satisfying the normalization conditions $\psi(q)=\ph(q)=1$,
and $w$ be their (constant) Wronskian. We set
$$
S(x,y):=\psi(x)\ph(y)-\psi(y)\ph(x),\qquad
C(x,y):=\psi'(x)\ph(y)-\psi(y)\ph'(x). $$
Obviously, $\omega=C(x,x)$ and $\bfrac{d}{dx} S(x,y)=C(x,y).$

Assume first that $z>0$  -- the case $z=0$
can be considered analogously. We find the function $G_z(x)$ for $q=0$ as defined in (\ref{Gz}).
We need the bounded continuous solution of the problem (\ref{2.6}), (\ref{2.7}) and 
(\ref{2.10}). For this conisder a candidate solution of the form 
$$G_z(x)=\mu \psi(x){\bold 1}_{(-\infty,0)}(x)
+\bfrac{2\la}{\om} S(x,z){\bold 1}_{[0,z]}(x)
+\eta \ph(x){\bold 1}_{(0,\infty)}(x), $$
where $\mu$ and $\eta$ are unknown constants. The continuity condition at point $z>0$ and condition (\ref{2.7}) hold in view of the structure of the function $G_z$. The continuity condition at $0$ implies
$$\mu \psi(0)=\bfrac{2\la}{\om} S(0,z)+\eta \ph(0). $$
From (\ref{2.10}) it follows that
$$\be\Big(\eta \ph'(0)+\bfrac{2\la}{\om} C(0,z)\Big)-(1-\be)\mu\psi'(0)
=\ga\mu \psi(0).$$
Taking into account the normalization conditions $\ph(0)=\psi(0)=1$, we have
\begin{align} 
& \mu-\eta=\bfrac{2\la}{\om}(\ph(z)-\psi(z)), \label{2.12} \\
&\mu((1-\be)\psi'(0)+\ga)-\eta \be \ph'(0)=\bfrac{2\la\be}{\om} C(0,z).
\label{2.13}
\end{align}
The solution of these algebraic equations are
\begin{align}
&\mu=\bfrac{2\la\be \ph(z)}{\be\om+(1-2\be)\psi'(0)+\ga},\label{2.14}\\
&\eta =\bfrac{2\la\be \ph(z)}{\be\om+(1-2\be)\psi'(0)+\ga}
+\bfrac{2\la}{\om}(\ph(z)-\psi(z)).\label{2.15}
\end{align}
Obviously, $G_z(0)=\mu.$ Inverting the Laplace transform with respect to $\ga$ yields
\begin{align}
\bfrac{\partial}{\partial z}\bfrac{\partial}{\partial v}\bE_0
\Big\{\exp\Big(&-\int_0^\tau f(W_\be(s))\,ds\Big);
\ell_\be(\tau,0)\le v, W_\be(\tau)<z\Big\}\nonumber\\
&=2\la\be\ph(z) {\rm e}^{-v(\be\om+(1-2\be)\psi'(0))}. \label{2.16}
\end{align}
For $f\equiv 0$ the fundamental solutions of (\ref{2.5}) with the properties
$\psi_0(0)=\ph_0(0)=1$ are $\psi_0(y)={\rm e}^{y\sqrt{2\la}}$,
$\ph_0(y)={\rm e}^{-y\sqrt{2\la}}$ and $w_0=2\sqrt{2\la}$ is their Wronskian.
Then for $v>0$
\begin{equation}
\bfrac{\partial}{\partial z}\bfrac{\partial}{\partial v}
\bP_0(\ell_\be(\tau,0)\le t, W_\be(\tau)<z)= 2\la\be
{\rm e}^{-(v+z)\sqrt{2\la}}. \label{2.17}
\end{equation}
Dividing (\ref{2.17}) by (\ref{1.10}) gives  (\ref{2.2}). Dividing (\ref{2.16}) by (\ref{2.17}) we have for the left hand side of (\ref{2.4})
\begin{align}
\bE_0^z\Big\{\exp\Big(-\int_0^\tau & f(W_\be(s))\,ds\Big)\Big|
\ell_\be(\tau,0)=v\Big\}\nonumber\\
&=\ph(z) {\rm e}^{(z+v)\sqrt{2\la}} {\rm e}^{-v(\be\om+(1-2\be)\psi'(0))}. \label{2.18}
\end{align}
Next we compute the analogous expression for the function $f_+$ in terms of
functions $\ph$ and $\psi$.
Since formula (\ref{2.18}) was obtained for arbitrary nonnegative piecewise
continuous function $f$ in terms of fundamental solutions of equation (\ref{2.5}),
it is possible to use these formulas for the function $f_+$. To do this
we express $\psi_+$, $\ph_+$, the fundamental solutions of equation (\ref{2.5})
with the function $f_+$ in place of $f$, via the fundamental solutions
$\psi$, $\ph$ for arbitrary $q$. We compute $\psi_+$ as an increasing continuous function with continuous first
derivative, satisfying the condition $\psi_+(q)=1$. We also compute
$\ph_+$ as a decreasing function with the condition $\ph_+(q)=1$. As a result,
$$
\psi_+(y)=\begin{cases} {\rm e}^{(y-q)\sqrt{2\la}}, & y\le q,\\
\bfrac{\sqrt{2\la}-\ph'(q)}{w}\psi(y)
+\bfrac{\psi'(q)-\sqrt{2\la}}{w}\ph(y), & q\le y,\end{cases} $$
$$\ph_+(y)=\begin{cases} \Big(\bfrac 12+\bfrac{\ph'(q)}{2\sqrt{2\la}}\Big)
{\rm e}^{(y-q)\sqrt{2\la}}+\Big(\bfrac 12-\bfrac{\ph'(q)}{2\sqrt{2\la}}\Big)
{\rm e}^{(q-y)\sqrt{2\la}}, & y\le q,\\
\ph(y), & q\le y.\end{cases} $$
The Wronskian of $\psi_+ $, $\ph_+ $ is
$$w_+=\psi_+'(q)-\ph_+'(q)=\sqrt{2\la}-\ph'(q).$$
Substituting $\psi_+$ and $\ph_+ $ in place of $\psi$ and $\ph$ for $q=0$
in (\ref{2.18}) and simplifying yield 
\begin{align}
\bE_0^z\Big\{\exp\Big(&-\int_0^\tau
 f_+(W_\be(s))\,ds\Big)\Big|\ell_\be(\tau,0)=v\Big\}\nonumber\\
&={\rm e}^{z\sqrt{2\la}}\ph(z) {\rm e}^{\be(\sqrt{2\la}+\ph'(0))}.\label{2.19}
\end{align}

We carry out the computation for $ f_-$. By the continuity of
the derivative at $q$ the fundamental solutions of
equation (\ref{2.5}) with $f_-$ in place of $f$ are
$$\psi_-(y)=\begin{cases} \psi(y), & y\le q,\\
\Big(\bfrac12+\bfrac{\psi'(q)}{2\sqrt{2\la}}\Big) {\rm e}^{(y-q)\sqrt{2\la}}
+\Big(\bfrac12-\bfrac{\psi'(q)}{2\sqrt{2\la}}\Big) {\rm e}^{(q-y)\sqrt{2\la}}, & q\le y,
\end{cases} $$
$$\ph_-(y)=\begin{cases} \bfrac{\sqrt{2\la}+\psi'(q)}{w}\ph(y)-\bfrac{\sqrt{2\la}+\ph'(q)}{w}\psi(y),
& y\le q,\\
{\rm e}^{(q-y)\sqrt{2\la}}, & q\le y\end{cases} $$
and the Wronskian  is
$$ w_-=\psi_-'(q)-\ph_-'(q)=\psi'(q)+\sqrt{2\la.} $$
Substituting $\psi_-$ and $\ph_-$ with $q=0$ in place of $\psi$ and $\ph$ in
(\ref{2.18}) and simplifying gives
\begin{equation}
\bE_0^z\Big\{\exp\Big(-\int_0^\tau f_-(W_\be(s))\,ds\Big)\Big|
\ell_\be(\tau,0)=v\Big\}={\rm e}^{(1-\be)v(\sqrt{2\la}-\psi'(0))}.\label{2.20}
\end{equation}
Since (\ref{2.18}) is equal to the product of (\ref{2.19}) by (\ref{2.20}), 
this proves (\ref{2.4})
and the Markov property at point zero holds.  The case  $q>z>0$ is studied analogously. We leave the details to the reader but point out the main formula
$$
\bE_0\Big\{\exp\Big(-\int_0^\tau f_+(W_\be(s))\,ds\Big)\Big|
\ell_\be(\tau,q)=v, W_\be(\tau)=z\Big\}$$
\begin{equation}
=\begin{cases} {\rm e}^{v(\ph'(q)+\sqrt{2\la})/2}, & 0<z\le q,\\
{\rm e}^{(z-q)\sqrt{2\la}}\ph(z) {\rm e}^{v(\ph'(q)+\sqrt{2\la})/2}, & 0<q\le z.\end{cases} 
\label{2.34}
\end{equation}
Also the details of the analysis of the case $q<0$ and $z>0$ are omitted. 

It remains to characterize the generating operators of these Markov processes. Hence  we consider  for $\eta>0$ the Laplace transforms
\begin{equation}\label{u00}
 u_{\pm}(h,v):=\bE_0^z\big\{{\rm e}^{-\eta\ell_\be(\tau,\pm h)}\big|
\ell_\be(\tau,0)=v\big\},\qquad\qquad h>0.
\end{equation}
To compute the function $u_+(h,v)$ we can use (\ref{2.19}), where the Dirac
$\del$-function at point $h$ multiplied by $\eta$ is taken instead of
the function $f_+$, i.e., in place of $f_+$
we consider the family of functions
$$
\Big\{\bfrac\eta\ep{\bold 1}_{[h,h+\ep)}(\cdot)\Big\}_{\ep>0}
$$
and pass to the limit as $\ep\downarrow 0$ in 
(\ref{2.19}).
This procedure is analogous to the one  used in the proof of
Theorem~3.1 Ch.~III of \cite{borodin13}. Clearly, we may as well consider $q+h$ instead of $h.$
If  $y\ge
q$ the fundamental solution $\ph_\delta$ of (\ref{2.5}) with the condition
$\ph_\delta(q)=1$  is the continuous decreasing solution of the
 problem
\begin{align} &\bfrac12\phi''(y)-\la\phi(y)=0,\qquad y\in(q,\infty)
\setminus\{q+h\},\nonumber\\
&\phi'(q+h+0)-\phi'(q+h-0)=2\eta\,\phi(q+h),\qquad\ph(q)=1,\nonumber
\end{align}
and is given by 
\begin{equation}
\ph_\delta(y)={\rm e}^{(q-y)\sqrt{2\la}}
+\bfrac{\eta\,{\rm e}^{-2h\sqrt{2\la}}({\rm e}^{(q-y)\sqrt{2\la}}-{\rm e}^{h\sqrt{2\la}}\,
{\rm e}^{-|y-q-h|\sqrt{2\la}})}{\sqrt{2\la}
(1+\frac{\eta}{\sqrt{2\la}}(1-{\rm e}^{-2h\sqrt{2\la}}))}.\label{2.21}
\end{equation}
Analogously, the  increasing solution is
\begin{equation}
\psi_\delta(y)={\rm e}^{(y-q)\sqrt{2\la}}+\bfrac{\eta}{\sqrt{2\la}}
\Big({\rm e}^{(y-q)\sqrt{2\la}}
-{\rm e}^{h\sqrt{2\la}}\,{\rm e}^{-|y-q-h|\sqrt{2\la}}\Big),\label{2.22}
\end{equation}
and the Wronskian is
\begin{equation}
\om_\delta=\psi_\del'(q)-\ph_\del'(q)
=\bfrac{2\sqrt{2\la}(1+\frac{\eta}{\sqrt{2\la}})}
{1+\frac{\eta}{\sqrt{2\la}}(1-{\rm e}^{-2h\sqrt{2\la}})}.\label{2.23}
\end{equation}

For $q-h$ we consider the problem
\begin{align} 
&\bfrac12\phi''(y)-\la\phi(y)=0,\qquad y\in(-\infty,q)\setminus\{q-h\},\nonumber\\
&\phi'(q-h+0)-\phi'(q-h-0)=2\eta\,\phi(q-h),\qquad\ph(q)=1.\nonumber
\end{align}
The solution is given by 
\begin{equation}
\ph_\delta(y)={\rm e}^{(q-y)\sqrt{2\la}}+\frac{\eta}{\sqrt{2\la}}
\Big({\rm e}^{(q-y)\sqrt{2\la}}
-{\rm e}^{h\sqrt{2\la}}\,{\rm e}^{-|y-q+h|\sqrt{2\la}}\Big),\label{2.24}
\end{equation}
\begin{equation}
\psi_\delta(y)={\rm e}^{(y-q)\sqrt{2\la}}
+\frac{\eta\,{\rm e}^{-2h\sqrt{2\la}}({\rm e}^{(y-q)\sqrt{2\la}}-{\rm e}^{h\sqrt{2\la}}\,
{\rm e}^{-|y-q+h|\sqrt{2\la}})}{\sqrt{2\la}
(1+\frac{\eta}{\sqrt{2\la}}(1-{\rm e}^{-2h\sqrt{2\la}}))},\label{2.25}
\end{equation}
and the Wronskian is as in (\ref{2.23}).

For $q=0$ and $0<h<z$ we have $\psi'_\del(0)=\sqrt{2\la}$ and
$$\ph_\delta(z)
=\bfrac{{\rm e}^{-z\sqrt{2\la}}}{1+\bfrac{\eta}{\sqrt{2\la}}(1-{\rm e}^{-2h\sqrt{2\la}})}. $$
Therefore we find by substituting (\ref{2.21}) in (\ref{2.19}) that $u_+$ as defined in (\ref{u00}) takes the form 
$$ u_+(h,v)=\bfrac1{1+\bfrac{\eta}{\sqrt{2\la}}(1-{\rm e}^{-2h\sqrt{2\la}})}
\exp\Big(-\bfrac{2v\be\eta {\rm e}^{-2\sqrt{2\la} h}}{1+
\bfrac\eta{\sqrt{2\la}}(1-{\rm e}^{-2h\sqrt{2\la}})}\Big), $$
and, hence,
$$\bE_0^z\big\{{\rm e}^{-\eta\ell_\be(\tau,0+)}\big|\ell_\be(\tau,0)=v\big\}
= {\rm e}^{-2\be v\eta}.$$
Similarly, substituting the derivative of (\ref{2.25}) in (\ref{2.20}), we have
$$ u_-(h,v)=
\exp\Big(-\bfrac{2v(1-\be)\eta {\rm e}^{-2\sqrt{2\la} h}}{1+
\bfrac\eta{\sqrt{2\la}}(1-{\rm e}^{-2h\sqrt{2\la}})}\Big), $$
and
$$\bE_0^z\big\{{\rm e}^{-\eta\ell_\be(\tau,0-)}\big|\ell_\be(\tau,0)=v\big\}
= {\rm e}^{-2(1-\be) v\eta}.$$
Therefore, under the condition $\ell_\be(\tau,0)=v$ the local time $\ell_\be(\tau,0+)$
equals $2\be v$ and  $\ell_\be(\tau,0-)$
equals $2(1-\be) v$. 
Thus also these calculations prove that the local time of the SBM is discontinuous at zero, and (\ref{121})  and (\ref{122}) hold.


To conclude the proof it remains to find the generating operators of 
$V_1$, $V_2$ and $V_3$. For this aim  we first
compute for $h\ge 0 $, $v>0 $ and $\eta>0$ the expressions for
\begin{align} &u_1(h,v):=\bE_0^z\big\{{\rm e}^{-\eta\ell(\tau,q+h)}\big|\ell(\tau, q)=v\big\},
\qquad 0<z\le q,\nonumber\\
&u_2(h,v):=\bE_0^z\big\{{\rm e}^{-\eta\ell(\tau,q+h)}\big|\ell(\tau, q)=v\big\},
\qquad 0<q<q+h\le z,\nonumber\\
&u_3(h,v):=\bE_0^z\big\{{\rm e}^{-\eta\ell(\tau, q-h)}\big|\ell(\tau,q)=v\big\},
\qquad q-h<q<0<z.\nonumber
\end{align}
These functions with respect to $\eta $ are the Laplace transforms of the transition
functions of $V_1$, $V_2$, and $V_3,$ respectively. Therefore they uniquely
determine the generating operators of the processes.

To compute $u_1(h,v)$ and $u_2(h,v)$ we can use (\ref{2.34}), where
the Dirac $\del$-function at point $q+h$, multiplied by $\eta$, is taken instead
of the function $f_+$.
Therefore we find by substituting the function (\ref{2.21}) in (\ref{2.34})
that
$$ u_1(h,v)=\exp\Big(-\bfrac{2v\eta {\rm e}^{-2\sqrt{2\la} h}}{1+
\bfrac\eta{\sqrt{2\la}}(1-{\rm e}^{-2h\sqrt{2\la}})}\Big). $$
$$ u_2(h,v)=\bfrac1{1+\bfrac{\eta}{\sqrt{2\la}}(1-{\rm e}^{-2h\sqrt{2\la}})}
\exp\Big(-\bfrac{2v\eta {\rm e}^{-2\sqrt{2\la} h}}{1+
\bfrac\eta{\sqrt{2\la}}(1-{\rm e}^{-2h\sqrt{2\la}})}\Big), $$
and
$$ u_3(h,v)=\exp\Big(-\bfrac{2v\eta {\rm e}^{-2\sqrt{2\la} h}}
{1+\bfrac\eta{\sqrt{2\la}}(1-{\rm e}^{-2h\sqrt{2\la}})}\Big). $$
Since $u_i(h,v)$, $i=1,2,3$, does not depend on $q$, the processes
$V_i\, i=1,2,3,$  are time homogeneous.
Now it can be checked that  $u_i, \, i=1,2,3,$ satisfy the equations
$$
\bfrac{\partial u_1}{\partial h}=\BL_1u_1,\qquad\bfrac{\partial u_2}
{\partial h}=\BL_2u_2,\qquad\bfrac{\partial u_3}
{\partial h}=\BL_3 u_3, 
$$ 
where the operators $\BL_i,\,i=1,2,3,$ are as given in Theorem \ref{th 2.1}. This completes the proof of the theorem. 
\hfill$\square$

\medskip
\noindent   
 {\bf Acknowledgements.} This research was partially supported by a grant from Magnus Ehrnrooths stiftelse, Finland, and the grant SPbU-DGF 6.65.37.2017. We thank  Endre Cs\'aki and Antonia  F{\"o}ldes  for posing the problem of finding the law of the supremum of the local time of a skew Brownian motion.


\end{document}